\documentclass[
final
]{dmtcs-episciences}

\usepackage[utf8]{inputenc}

\usepackage[round]{natbib}

\usepackage{amsmath,amsthm,amssymb}
\usepackage{mathrsfs}
\usepackage[shortlabels]{enumitem}
\usepackage{stmaryrd}
\usepackage{shuffle}
\usepackage{multicol} 
\usepackage{mathtools}
\usepackage{commath}
\usepackage{ytableau}
\usepackage{caption}
\usepackage{subcaption}
\usepackage{listofitems}
\usepackage{tikz-cd}
\usepackage{comment}
\usepackage{mathdots}

\newtheorem{theorem}{Theorem}[section]

\newtheorem{prop}[theorem]{Proposition}
\newtheorem{lemma}[theorem]{Lemma}
\newtheorem{question}[theorem]{Question}
\newtheorem{claim}[theorem]{Claim}

\theoremstyle{definition}

\usepackage{tikz}
\usetikzlibrary{calc,backgrounds}

\newcommand{\N}{\mathbb{N}}
\author[Kyle Celano et al.]{Kyle Celano\affiliationmark{1}
  \and Abigail Ollson\affiliationmark{2}
  \and Niraj Velankar\affiliationmark{3} \and Jun Yan\affiliationmark{4}}
\title{An Erd\H{o}s--Szekeres type result for words with repeats}
\affiliation{
  Wake Forest University, Winston-Salem, USA\\
  Keele University, Keele, UK\\
  University of Greifswald, Greifswald, Germany\\
  University of Oxford, Oxford, UK}
\begin{document}
\publicationdata{vol. 28:1, Permutation Patterns 2025}{2026}{5}{10.46298/dmtcs.16801}{2025-10-28; 2025-10-28; 2026-05-12; 2026-05-13}{2026-05-14}

\maketitle

\begin{abstract}
\vspace{0.25cm}

We prove an Erd\H{o}s--Szekeres type result for finite words over $\mathbb{N}$ with repeated values. Specifically, we define a \emph{repeat} in a word to be an occurrence of a value which is not its first occurrence. We define an occurrence of a \emph{pattern} $\pi$ in a word $w$ to be a (not necessarily consecutive) subword of $w$ that is order isomorphic to $\pi$.  In this note, we show that every word with $kn^6+1$ repeats contains one of the following patterns: $0^{k+2}$, $0011\cdots nn$, $nn\cdots1100$, $012 \cdots n012 \cdots n$, $012 \cdots nn\cdots 210$, $n\cdots 210012\cdots n$, $n\cdots 210n\cdots 210$. Moreover, when $k=1$, we show that this is best possible by constructing a word with $n^6$ repeats that does not contain any of these patterns.
\end{abstract}

\section{Introduction}
\label{sec:in}
Let $\N=\{0,1,2,\dots\}$. A \emph{word} of length $n$ is an element $w=w_1w_2\cdots w_n\in\N^n$. For example, 376885 is a word of length 6. A word $w\in\N^n$ can also be represented on a grid with points at $(i, w_i)$ for every $i \in[n]$, where $[n] = \{1, 2, \dots, n\}$. For example, the word $13043134$ is shown in Figure~\ref{fig: 13043134 word}.

\begin{figure}[h]
  \centering    
  \resizebox{0.5\textwidth}{!}{%
  \begin{tikzpicture}
        \draw[step=1cm,gray,very thin] (0,0) grid (9,5);
        \draw[thick,->] (0,0) -- (9,0);
        \draw[thick,->] (0,0) -- (0,5);
        \foreach \y in {0, 1, 2, 3, 4}
            \node at (-0.2, \y) {\y};
        \foreach \x in {1, 2, 3, 4, 5, 6, 7, 8}
            \node at (\x, -0.4) {\x};
        \foreach \x/\y in {1/1, 2/3, 3/0, 4/4, 5/3, 6/1, 7/3, 8/4}
            \fill[black] (\x,\y) circle (0.11cm);
    \end{tikzpicture}
    }%
  \caption{The word 13043134.}
  \label{fig: 13043134 word}
proof\end{figure}

We define a \emph{subword} $w'$ of a word $w$ to be a word obtained by taking a subsequence of letters in $w$ which are not necessarily consecutive, for example 3313 is a subword of the word 13043134. The \emph{standardization} of a word $w \in \mathbb{N}^n$ is defined to be the word obtained by replacing every occurrence of the smallest value in $w$ by 0, the next smallest by 1, and so on. The standardization of a word is called a \emph{pattern}\footnote{These are also called Cayley permutations; see~\cite{Mor1984}}. For example, the standardization of the word 296893 is the pattern 042341. A word $w$ \emph{contains} a pattern $\pi$ if there exists a subword of $w$ which standardizes to $\pi$, otherwise $w$ \emph{avoids} $\pi$. As the standardization of 3313 is 1101, the word 13043134 contains the pattern 1101.
This definition of pattern containment is a generalisation of classical pattern avoidance in permutations,
and has also been used to study pattern containment in words, see for example 
\cite{Burstein2003} on counting the number of occurrences of patterns in words.

A word $w=w_1 w_2 \cdots w_n$ is called \emph{non-decreasing} if $w_i \leq w_j$ for all $1\leq i<j\leq n$, and is called  \emph{strictly increasing} if $w_i < w_j$ for all $1\leq i<j\leq n$. For example, 00122 is non-decreasing and 0123 is strictly increasing. \emph{Non-increasing} and \emph{strictly decreasing} words are defined similarly. 

The Erd\H{o}s--Szekeres Theorem, stated below in Theorem~\ref{thm:ES}, is a landmark result in combinatorics with many variants and generalisations, and wide-ranging applications in areas like discrete geometry, graph theory, and computer science.

\begin{theorem}[Erd\H{o}s--Szekeres Theorem, \cite{Erdos1935}]
  \label{thm:ES}
  Let $r,s \in\N$. Every word of length $rs + 1$ contains a non-decreasing subword of length $r+1$ or a non-increasing subword of length $s+1$ (or both).
\end{theorem}

For permutations, the Erd\H{o}s--Szekeres Theorem implies that any permutation of size $n^2+1$ contains a strictly increasing or decreasing pattern of size $n+1$, or equivalently either the pattern $01\cdots n$ or $n\cdots 10$. 
For general words that can have repeated values, this is no longer true. 
There have been many variations of the Erd\H{o}s--Szekeres Theorem such as the containment of unimodal and more generally $k$-modal sequences in permutations by \cite{Chung1980}, \cite{Xu2024} and \cite{Gong2025}; finding words with non-decreasing or strictly increasing subwords of a given length by \cite{Schensted1961} and \cite{Itskovich2019}; or the containment of convex and concave patterns by \cite{Erdos1935} and \cite{Suk2016} for example.

This project began at the 2025 Permutation Patterns Pre-Conference Workshop, where we aimed to find an Erd\H{o}s--Szekeres type result for inversion sequences, specifically taking into account their structure and the number of repeated values.
For a positive integer $n$, we define a length $n$ \emph{inversion sequence} to be a word $w_1w_2 \cdots w_n \in \mathbb{N}^n$ where $w_i < i$ for each $i\in[n]$, or equivalently $w_1w_2 \cdots w_n$ is an element of $\{0\}\times \{0,1\}\times \cdots \times \{0,1,\dots,n-1\}$. 
For example, the inversion sequence $002135$ is shown in Figure~\ref{fig:002314 inversion sequence}. Due to the restriction on values in an inversion sequence, all of the points in an inversion sequence must be below the dashed line in Figure \ref{fig:002314 inversion sequence}. 

\begin{figure}[h]
  \centering
  \resizebox{0.4\textwidth}{!}{%
  \begin{tikzpicture}
        \draw[step=1cm,gray,very thin] (-1,0) grid (6,6);
        \draw[thick,->] (-1,0) -- (6,0);
        \draw[thick,->] (-1,0) -- (-1,6);
        \foreach \y in {0, 1, 2, 3, 4, 5}
            \node at (-1.2, \y) {\y};
        \foreach \x in {1, 2, 3, 4, 5, 6}
            \node at (\x-1, -0.4) {\x};
        \foreach \x/\y in {0/0, 1/0, 2/2, 3/1, 4/3, 5/5} {
          \node [font=\large] at (\x,\y) {$\bullet$};
        };    
        \draw[thick, dashed] (-0.5,0) -- (-0.5, 0.5) -- (0.5, 0.5) -- (0.5, 1.5) -- (1.5, 1.5) -- (1.5, 2.5) -- 
        (2.5, 2.5) -- (2.5, 3.5) -- (3.5, 3.5) -- (3.5, 4.5) -- (4.5, 4.5) -- (4.5, 5.5) -- (5.5, 5.5) -- (5.5, 6);
  \end{tikzpicture}
  }%
  \caption{The inversion sequence 002135.}
  \label{fig:002314 inversion sequence}
\end{figure}

In this note, we prove Theorem~\ref{prop: word with m repeats}, which identifies a collection of unavoidable patterns in any word with a given number of repeated values. This result can be applied to inversion sequences and will be used in an upcoming work which outlines an algorithm for enumerating sets of inversion sequences using generating trees. In the paper, this result is used to prove conditions for when the algorithm will terminate. 

We define the number of \emph{repeats} in a word to be the total number of occurrences of values which are not the first occurrence of that value. Equivalently, for a word $w=w_1\cdots w_n$, the number of repeats is number of $i\in [n]$ such that there exists $j<i$ with $w_j=w_i$. For example, 0000 has 3 repeats, and 00110 has 3 repeats as 0 is repeated twice and 1 is repeated once.

For a word $w=w_1 w_2 \cdots w_n$, let $w^{id}=w$, and let $w^{rev}=w_n \cdots w_2 w_1$ be the \emph{reverse} of $w$, e.g. $(2341)^{rev}=1432$. We show that the types of unavoidable patterns are the constant patterns, patterns of the form $001122 \cdots nn$ and their reverses, and patterns of the form $u^{e_1}u^{e_2}$ where $u=01\cdots n$ and $e_i\in \{id,rev\}$. These three types are defined as types~\ref{1},~\ref{2} and~\ref{3} in Theorem~\ref{prop: word with m repeats}, respectively. 

\begin{theorem}
\label{prop: word with m repeats}
Let $k,n\in\mathbb{N}$. Every word with $kn^6+1$ repeats contains one of the following patterns.
\begin{enumerate}[label={\upshape\alph*)}]
    \item\label{1} $0^{k+2}$.
    \item\label{2} $(0011\cdots nn)^e$ for $e \in \{id, rev \}$.
    \item\label{3} $(012 \cdots n)^{e_1}(012 \cdots n)^{e_2}$ for $e_1, e_2 \in \{id, rev \}$.
\end{enumerate}
\end{theorem}

When $k=1$, Theorem~\ref{prop: word with m repeats} is best possible as shown by the following result. See Section~\ref{sec: conc} for a discussion when $k>1$.
\begin{prop}
\label{prop: bound is tight}
There exists a word with $n^6$ repeats that does not contain any of the following patterns.
\begin{enumerate}[label={\upshape \alph*)}]
    \item\label{eg-1} $0^3$.
    \item\label{eg-2} $(0011\cdots nn)^e$ for $e \in \{id, rev \}$.
    \item\label{eg-3} $(012 \cdots n)^{e_1}(012 \cdots n)^{e_2}$ for $e_1, e_2 \in \{id, rev \}$.
\end{enumerate}
\end{prop}

The asymmetric versions of both Theorem~\ref{prop: word with m repeats} and Proposition~\ref{prop: bound is tight}, where the lengths of the patterns in~\ref{2} and~\ref{3} are allowed to differ like in Theorem~\ref{thm:ES}, can also be easily obtained by modifying our proof (see Theorem~\ref{thm:asym}). 

We will prove Theorem~\ref{prop: word with m repeats} and Proposition~\ref{prop: bound is tight} in Section~\ref{sec: first proof} and Section~\ref{sec: second proof}, respectively. In Section~\ref{sec: conc}, we will discuss some related problems and future research directions. 


\section{Unavoidable patterns}\label{sec: first proof}
In this section, we prove Theorem~\ref{prop: word with m repeats}, which states that any word with $kn^6+1$ repeats must contain one of the patterns in~\ref{1},~\ref{2} and~\ref{3}.

\begin{proof}[of Theorem~\ref{prop: word with m repeats}]
Let $w$ be a word with $kn^6+1$ repeats. If $w$ contains a value that occurs $k+2$ times, then it contains the pattern $0^{k+2}$, so \ref{1} holds. Now supppose that every value appears at most $k+1$ times. Then, there are at least $n^6+1$ distinct values in $w$ each with at least one repeat. Fix any $n^6+1$ distinct values in $w$ with repeats, and let $w'$ be the subword of $w$ consisting of the first two occurrences of each of these values in $w$. In particular, $w'$ has length $2n^6+2$.

Let $a_0 a_1\cdots a_{n^6}$ be the subword of $w'$ consisting of the first occurrence of each of the $n^6+1$ distinct values in $w'$. By Theorem~\ref{thm:ES} (Erd\H os--Szekeres), there is a subword $a'_0 a'_1\cdots a'_{n^3}$ of $a_0 a_1\cdots a_{n^6}$ that is monotone.

We first attempt to find a pattern in~\ref{3} using Claim~\ref{claim:main} below. 

\begin{claim}\label{claim:main}
If there exists an index $0\leq j\leq n^3-n^2$, such that the second occurrences in $w'$ of the values $a'_j,a'_{j+1},\ldots,a'_{j+n^2}$ are all after the first occurrence of the value $a'_{j+n^2}$, then $w'$ contains a pattern in~\ref{3}.
\end{claim}
\begin{proof}[of Claim~\ref{claim:main}]
Suppose the repeats of $a'_j, a'_{j+1},\ldots, a'_{j+n^2}$ form the subword $b_0b_1\cdots b_{n^2}$ in $w'$. By Theorem~\ref{thm:ES} (Erd\H os--Szekeres), $b_0\cdots b_{n^2}$ contains a monotone subword $b_0'\cdots b_{n}'$. Let $a_0'' \cdots a_n''$ be the subword of $a'_j a'_{j+1} \cdots a'_{j+n^2}$ consisting of the values in $b_0' \cdots b_n'$, and note that it is also monotone. Therefore, $a_0'' \cdots a_n'' b_0' \cdots b_n'$ is a subword of $w'$ forming a pattern in~\ref{3}.
\renewcommand{\qedsymbol}{$\boxdot$}
\end{proof}
\renewcommand{\qedsymbol}{$\square$}

If the hypothesis in Claim~\ref{claim:main} holds for any $j$ of the form $j=(t-1)n^2$ with $t\in[n]$, then we are done, so let us assume the opposite. Hence, for each $t\in[n]$, there exists some $(t-1)n^2\leq i_t\leq tn^2-1$, such that both occurrences of $a'_{i_t}$ appear before the first occurrence of $a'_{tn^2}$ in $w'$ (see Figure~\ref{fig:case 1}). Therefore, $a_{i_1}'a_{i_1}'\cdots a_{i_n}'a_{i_n}'a_{n^3}'a_{n^3}'$ is a subword of $w'$. Moreover, this forms a pattern in \ref{2} because $a_{i_1}' \cdots a_{i_n}' a_{n^3}'$ is a subword of $a'_0  a'_1 \cdots a'_{n^3}$, and hence monotone.
\end{proof}
\begin{figure}[h]
    \centering

\[\begin{tikzpicture}[xscale=.9,yscale=.75]
     \node at (-.5,.5) {$w'=$};
    \draw (0,0) rectangle (12,1);
    \draw (4,0)--(4,1);
    \draw (8,0)--(8,1);
    \node at (12.5,.5) {$\cdots$};

    \node at (.5,.5) {$a_0'$};
    \node at (4.6,.5) {$a_{n^2}'$};
    \node at (8.65,.5) {$a_{2n^2}'$};
    \node (a1) at (1.5,-1) {$a_{i_1}'$};
    \node (b1) at (3,-1) {$a_{i_1}'$};
    \draw[->] (a1)--(1.5,.5);
    \draw[->] (b1)--(3,.5);
    \node (a2) at (5.5,-1) {$a_{i_2}'$};
    \node (b2) at (7,-1) {$a_{i_2}'$};
    \draw[->] (a2)--(5.5,.5);
    \draw[->] (b2)--(7,.5);
    \node (a3) at (9.6,-1) {$a_{i_3}'$};
    \node (b3) at (11.1,-1) {$a_{i_3}'$};
    \draw[->] (a3)--(9.6,.5);
    \draw[->] (b3)--(11.1,.5);
\end{tikzpicture}\]
    \caption{An occurrence of a pattern type \ref{2} in $w'$.}
    \label{fig:case 1}
\end{figure}

\section{Extremal construction}
\label{sec: second proof}
In this section, we construct an example proving Proposition~\ref{prop: bound is tight}, which shows that when $k=1$, Theorem~\ref{prop: word with m repeats} is best possible. In order to describe our construction, we first need the following definitions. For convenience, all words in this section do not contain the value 0. 

Let $\pi = \pi_1 \cdots \pi_n$ be a word with length $n$ and maximum value $h= \text{max}\{\pi_i\mid i \in [n] \}$. In the same way, let $\sigma= \sigma_1 \cdots \sigma_m$ be a word with length $m$ and maximum value $\ell = \text{max} \{\sigma_j\mid j \in [m] \}$.
\begin{itemize}
    \item The \emph{concatenation} of $\pi$ and $\sigma$, denoted $\pi\cdot\sigma$, is the word $$\pi_1 \cdots \pi_n \sigma_1 \cdots \sigma_m.$$
    \item The \emph{direct sum} of $\pi$ and $\sigma$, denoted $\pi \oplus \sigma$, is the word \[\pi_1 \cdots \pi_n b_1 b_2\cdots b_m,\] such that $b_j = \sigma_j + h$ for every $j\in[m]$.
    Using the graphical representation in Figure~\ref{fig: 13043134 word}, this can be viewed as a grid where the bottom left corner of the grid represents $\pi$, the top right corner represents $\sigma$, and the other two corners are empty, as shown in Figure~\ref{fig: direct sum}. For example, $31422 \oplus 4132 =314228576$ and $21 \oplus 21 \oplus 21=214365$.
    \item Similarly, the \emph{skew sum} of $\pi$ and $\sigma$, denoted $\pi \ominus \sigma$, is the word 
    \[a_1a_2\cdots a_n\sigma_1 \cdots \sigma_m,\]
    such that $a_i = \pi_i + \ell$ for every $i\in[n]$.
    A graphical representation is given in Figure \ref{fig: skew sum}. For example, $2413 \ominus 121 = 4635121$ and $12 \ominus 12 \ominus 12=563412$. 
\end{itemize} 

\begin{figure}[h]
    \centering
    \begin{subfigure}{0.4\textwidth}
      \centering
      \begin{tikzpicture}
      \tikzstyle{every node}=[font=\LARGE]
      \draw  (0.25,9.5) rectangle (2.25,7.5);
      \draw [ fill=gray!40] (1.25,8.5) rectangle (0.25,9.5);
      \draw [  fill=gray!40 ] (1.25,8.5) rectangle (2.25,7.5);
      \node [font=\LARGE] at (0.75,8) {$\pi$};
      \node [font=\LARGE] at (1.75,9) {$\sigma$};
      \end{tikzpicture}
      \subcaption{The direct sum $\pi \oplus \sigma$.}
      \label{fig: direct sum}
    \end{subfigure}
    \centering
    \begin{subfigure}{0.4\textwidth}
      \centering
      \begin{tikzpicture}
      \tikzstyle{every node}=[font=\LARGE]
      \draw  (0.25,9.5) rectangle (2.25,7.5);
      \draw [fill=gray!40] (1.25,7.5) rectangle (0.25,8.5);
      \draw [fill=gray!40] (1.25,9.5) rectangle (2.25,8.5);
      \node [font=\LARGE] at (0.75,9) {$\pi$};
      \node [font=\LARGE] at (1.75,8) {$\sigma$};
      \end{tikzpicture}
    \subcaption{The skew sum $\pi \ominus \sigma$.}
    \label{fig: skew sum}
    \end{subfigure}
    \centering
    \caption{The direct sum and skew sum of two words.}
    \label{fig: direct and skew sums}
\end{figure}

Note that both direct sums and skew sums are associative operations. We will also use $w^{\oplus n}=w \oplus w \oplus \cdots \oplus w$ and $w^{\ominus n}=w \ominus w \ominus \cdots \ominus w$
to denote taking the direct sums or skew sums of $n$ copies of $w$.
With these definitions, we can now construct the word we use to prove Proposition~\ref{prop: bound is tight}. This is done in several steps. Throughout, let $n$ be a fixed positive integer.

Let $p$ denote the word $12\cdots n^2$ of length $n^2$ and let $t$ denote the word $1 2 \cdots n$ of length $n$. Furthermore, let $r$ denote the word $t^{\ominus n}$ of length $n^2$ formed by taking the skew sums of $n$ copies of $t$ (see also Figure~\ref{fig:$n$ disjoint increasing sequences of length $n$}).
Since $t=1^{\oplus n}$, we may also write $r$ as $(1^{\oplus n})^{\ominus n}$.
For example, for $n=3$ we have $r=123\ominus 123\ominus 123=789456123$.
Observe that any monotone subword of $r$ has length at most $n$.

\begin{figure}[h]
    \centering
    \begin{tikzpicture}
        \draw[step=1cm,black,very thin] (0, 0) grid (4,4);
        \draw (0.2, 3.2) -- (0.8, 3.8);
        \draw (1.2, 2.2) -- (1.8, 2.8);
        \draw (3.2, 0.2) -- (3.8, 0.8);
        \node at (2.55, 1.6) {$\ddots$};        
    \end{tikzpicture}
    \caption{The word $r$ formed by taking the skew sums of $n$ copies of the word $t=12\cdots n$.}
    \label{fig:$n$ disjoint increasing sequences of length $n$}
\end{figure}

With these words $p$ and $r$, we create a word $q$ of length $2n^4$ which is the main building block of our construction in Proposition~\ref{prop: bound is tight}. 
First, take the skew sums of $n$ copies of $r$ to create the length $n^3$ word \[r^{\ominus n}=r \ominus r \ominus \cdots \ominus r.\]
Then, take the direct sums of $n$ copies of this word to create the length $n^4$ word \[r' = (r^{\ominus n})^{\oplus n} = ((r \ominus \cdots \ominus r) \oplus \cdots \oplus (r \ominus \cdots \ominus r)),\] which is represented in Figure~\ref{fig:lots of r sums}.

\begin{figure}[h]
    \centering
    \resizebox{0.4\textwidth}{!}{
    \begin{tikzpicture}
        \draw[step=2cm,black,very thin] (0, 0) grid (6,6);
        \node[font=\LARGE] at (3, 3.2) {$\iddots$};   
        \foreach \x/\y in {0.4/1.6, 0.8/1.2, 1.6/0.4, 4.4/5.6, 4.8/5.2, 5.6/4.4}
            \node[font=\large] at (\x, \y) {$r$};
        \foreach \x/\y in {1.2/0.9, 5.2/4.9}
            \node[font=\tiny] at (\x, \y) {$\ddots$};
    \end{tikzpicture}
    }
    \caption{The word $r'=(r^{\ominus n})^{\oplus n}$.}
    \label{fig:lots of r sums}
\end{figure}

The word $q$ is obtained by concatenating $r'$ to the end of the skew sums of $n^2$ copies of $p$, as follows. 
\begin{equation*}
    \begin{split}
        q &= p^{\ominus n^2} \cdot r' = p^{\ominus n^2} \cdot (r^{\ominus n})^{\oplus n} \\
        &= (p \ominus \cdots \ominus p) \cdot ((r \ominus \cdots \ominus r) \oplus \cdots \oplus (r \ominus \cdots \ominus r))
    \end{split}
\end{equation*}
For example, for $n=2$ we have \[q=(p\ominus p\ominus p\ominus p)\cdot((r\ominus r)\oplus(r\ominus r)),\] 
and $q$ is depicted in Figure~\ref{fig:p and rs in q when n=2}, with the grid arranged so that copies of $p$ and $r$ are in the same rows if and only if they are words on the same set of values.
\begin{figure}[h]
    \centering
    \[q=\begin{array}{|c|c|c|c|c|c|c|c|}
    \hline
        p&&& &&&r&\\\hline
        &p&& &&&&r\\\hline
        &&p& &r&&&\\\hline
        &&&p &&r&&\\\hline 
    \end{array}\]
    \caption{The arrangement of $p$'s and $r$'s in $q$ when $n=2$.}
    \label{fig:p and rs in q when n=2}
\end{figure}

Note that the restriction of the word $q$ to any of
the rows shown in Figure~\ref{fig:p and rs in q when n=2} is a concatenation of $p$ and $r$. A pictorial representation of $p\cdot r$ when $n=4$ is shown in Figure~\ref{fig:p and r concat}.

\begin{figure}[h]
\[\begin{tikzpicture}[scale=.25]
\tikzset{every path/.style={very thick}}
    \draw (1,1)--(16,16);
    \draw (17,13)--(20,16);
    \draw (21,9)--(24,12);
    \draw (25,5)--(28,8);
    \draw (29,1)--(32,4);
    \tikzset{every path/.style={color=gray,thin}}
    \draw (1,1)-- (29,1);
    \draw (4,4)-- (32,4);
    \draw (5,5)--(25,5);
    \draw (8,8)--(28,8);
    \draw (9,9)--(21,9);
    \draw (12,12)-- (24,12);
    \draw (13,13)--(17,13);
    \draw (16,16)-- (20,16);
\end{tikzpicture}\]
    \caption{A representation of $p\cdot r$ when $n=4$.}
    \label{fig:p and r concat}
\end{figure}

We now prove a key property of the word $q$ in Lemma~\ref{claim: p and r concat containment}.

\begin{lemma}\label{claim: p and r concat containment}
    The word $q$ does not contain any of the following patterns: 
\begin{enumerate}[label={\upshape\alph*)}]
    \item\label{claim-1}  $0^3$,
    \item\label{claim-2} $0011$, $1100$,
    \item\label{claim-3} $(012 \cdots n)^{e_1}(012 \cdots n)^{e_2}$ for $e_1, e_2 \in \{id, rev \}$.
\end{enumerate}
\end{lemma}
\begin{proof}
We rule out these patterns one by one.

For~\ref{claim-1}, every number appears exactly twice in $q$, so the pattern $0^3$ does not appear. 

For~\ref{claim-2}, let $a,b\in[n^2]$ be distinct. Then, $a$ and $b$ each appears exactly once in $p^{\ominus n^2}$ and once in $(r^{\ominus n})^{\oplus n}$. Since $q$ is formed by concatenating these two words, the second occurrence of $a$ in $q$ is after the first occurrence of $b$, and vice versa. Therefore, neither $0011$ nor $1100$ can appear as a pattern. 

For~\ref{claim-3}, if any such pattern exists in $q$, then the $(012\cdots n)^{e_1}$ must appear entirely within $p^{\ominus n^2}$, and $(012\cdots n)^{e_2}$ must appear entirely within $(r^{\ominus n})^{\oplus n}$. We split into two cases depending on whether $e_1$ is $id$ or $rev$. 

\textbf{Case 1.}
If $e_1=id$ then $(012 \cdots n)^{e_1}=012 \cdots n$ is increasing, so it must be within a single copy of $p$, as across skew sums the entries are decreasing by definition. 
Thus, the occurrence of $(012 \cdots n)^{e_2}$ in $(r^{\ominus n})^{\oplus n}$ must be entirely within the copy of $r$ corresponding to this $p$. In Figure~\ref{fig:p and rs in q when n=2}, this copy of $r$ is the one located in the same row as the corresponding copy of $p$.
However, this would imply that this copy of $r$ contains a monotone subword of length $n+1$, which is a contradiction. 

\textbf{Case 2.}
If $e_1=rev$ then $(012 \cdots n)^{e_1}=n \cdots 210$ is decreasing, so the $n+1$ numbers forming this pattern must appear in pairwise distinct copies of $p$, as $p$ itself is increasing. Then, the $n+1$ numbers in the occurrence of $(012 \cdots n)^{e_2}$ in $(r^{\ominus n})^{\oplus n}$ must also each come from different copies of $r$.
Taking one value from each copy of $r$ in $q$ creates a subword which standardizes to $(1^{\ominus n})^{\oplus n}$ (see Figure~\ref{fig:lots of r sums}, taking one value from each copy of $r$), which is the reverse of the word $r$. As the longest monotone subword in $r$ has length $n$, the longest monotone subword in $(1^{\ominus n})^{\oplus n}$ also has length $n$. Hence, there is no monotone subword of length $n+1$ in $(r^{\ominus n})^{\oplus n}$ whose entries all come from different copies of $r$, which is a contradiction.
\end{proof} 

Finally, we can define the word $s$ we use to prove Proposition~\ref{prop: bound is tight}. This is obtained by first taking the direct sum of $n$ copies of the word $q$ to obtain the length $2n^5$ word \[q^{\oplus n} = q \oplus q \oplus \cdots \oplus q,\]
and then taking the skew sums of $n$ copies of this word to create the length $2n^6$ word \[s = (q^{\oplus n})^{\ominus n} = (q \oplus q \oplus \cdots \oplus q) \ominus  \cdots \ominus (q \oplus q \oplus \cdots \oplus q).\]
For example, for $n=2$, the word $s$ is depicted in Figure~\ref{fig: qs in s, n=2}.

\begin{figure}[h]
    \centering
\[s=\begin{array}{|c|c|c|c|}
\hline
    &q&& \\\hline
    q&&& \\\hline
    &&&q\\\hline
    &&q& \\\hline 
\end{array}\]
\def\q{\arraycolsep=1.4pt\begin{array}{cccccccc}
    p&&& &&&r&\\
    &p&& &&&&r\\
    &&p& &r&&&\\
    &&&p &&r&&\\
\end{array}}
\[s=\arraycolsep=1.4pt\begin{array}{|c|c|c|c|}
\hline
    &\q&& \\\hline
    \q&&& \\\hline
    &&&\q\\\hline
    &&\q& \\\hline 
\end{array}\]
    \caption{The arrangement of subwords in $s$ for $n=2$.}
    \label{fig: qs in s, n=2}
\end{figure}

Now we can prove Proposition~\ref{prop: bound is tight}.

\begin{proof}[of Proposition~\ref{prop: bound is tight}]
Let $s$ be the word of length $2n^6$ constructed above. We will prove that $s$ does not contain any of the patterns stated in the proposition. Clearly, $s$ does not contain a pattern of type \ref{eg-1} as each value appears exactly twice. 

If $s=(q^{\oplus n})^{\ominus n}$ contains a pattern $0011\cdots nn$, then the pattern must be entirely contained within a copy of $q^{\oplus n}$, because across skew sums the entries are decreasing from definition. Then, by the pigeonhole principle, one copy of $q$ in this $q^{\oplus n}$ must contain two distinct values that are part of this $0011\cdots nn$ pattern, and thus this $q$ contains a pattern $0011$. However, this is not possible by Lemma~\ref{claim: p and r concat containment}.

If $s=(q^{\oplus n})^{\ominus n}$ contains a pattern $nn\cdots1100$, then by the pigeonhole principle, one copy of $q^{\oplus n}$ must contribute two distinct entries to this pattern. However, these two distinct entries cannot come from two distinct copies of $q$ in this $q^{\oplus n}$, as they must be in decreasing order. Thus, they come from the same copy of $q$, which then contains a pattern $1100$. This is again impossible by Lemma~\ref{claim: p and r concat containment}, which proves that no pattern of type \ref{eg-2} is contained in $s$. 

Any patterns of type \ref{eg-3} in $s$ must be entirely contained within a single copy of $q$ that forms $s$ because for every number in $s$, both of its occurrences in $s$ lie in the same copy of $q$. But this is not possible as $q$ contains none of the patterns in \ref{eg-3} by Lemma~\ref{claim: p and r concat containment}.

This covers all possibilities, so none of the patterns in the statement of the proposition is contained in $s$.
\end{proof}

\section{Concluding remarks}\label{sec: conc}
In this paper, we proved in Theorem~\ref{prop: word with m repeats} the patterns that must be contained in a word with $kn^6+1$ repeated values.
We also showed in Proposition~\ref{prop: bound is tight} that Theorem~\ref{prop: word with m repeats} is best possible when $k=1$. For $k>1$, by changing the word $t$ in the construction of Proposition~\ref{prop: bound is tight} from $12\cdots n$ to $1^k2^k\cdots n^k$ while keeping everything else the same, essentially the same proof shows that the resulting word $s$ consists of exactly $k+1$ occurrences of every number in $[n]$, and thus $kn^6$ repeats, but does not contains any of the seven patterns $0^{k+2}$, $(0^{k+1}1^{k+1}\cdots n^{k+1})^e$, and $(012 \cdots n)^{e_1}(012 \cdots n)^{e_2}$, where $e,e_1,e_2\in\{id,rev\}$. This suggests that Theorem~\ref{prop: word with m repeats} is likely not optimal when $k>1$, and it remains open to determine the minimum number of repeats needed to guarantee one of the seven patterns considered there.
\begin{question}
For integers $n\geq 1$ and $k>1$, what is the smallest $m=m(n,k)$ such that every word with $m$ repeats contain one of the following patterns?
\begin{enumerate}[label={\upshape\alph*)}]
    \item $0^{k+2}$.
    \item $(0011\cdots nn)^e$ for $e \in \{id, rev \}$.
    \item $(012 \cdots n)^{e_1}(012 \cdots n)^{e_2}$ for $e_1, e_2 \in \{id, rev \}$.
\end{enumerate}
\end{question}

By adapting our proof of Theorem~\ref{prop: word with m repeats} and Proposition~\ref{prop: bound is tight}, we can easily obtain the following asymmetric version.
\begin{theorem}\label{thm:asym}
Let $k,x_1,x_2,y_1,y_2,z_1,z_2$ be positive integers. Every word with $kx_1x_2y_1y_2z_1z_2+1$ repeats contains one of the following patterns.
\begin{enumerate}[label={\upshape\alph*)}]
    \item $0^{k+2}$.
    \item $0011\cdots x_1x_1$ or $x_2x_2\cdots1100$.
    \item $012 \cdots y_1012 \cdots y_1$ or $012 \cdots y_2y_2\cdots 210$ or $z_1 \cdots 210012\cdots z_1$ or $z_2 \cdots 210z_2\cdots210$.
\end{enumerate}
Moreover, this result is best possible when $k=1$, as there is a word with $x_1x_2y_1y_2z_1z_2$ repeats avoiding all of the patterns above. 
\end{theorem}

While in this note we only considered unavoidable patterns for words with a given number of repeats, it is also natural to consider the same problem for a more restricted family of words of length $(k+1)N$ consisting of exactly $k+1$ occurrences of each number in $[N]$, and thus $kN$ repeats. In this setting, as mentioned above, we can easily adapt our proofs of Theorem~\ref{prop: word with m repeats} and Proposition~\ref{prop: bound is tight} to show the following.
\begin{prop}
For every integer $k\geq 1$, any word consisting of exactly $k+1$ occurrences of each number in $[n^6+1]$ must contain one of the following patterns.
\begin{enumerate}[label={\upshape\alph*)}]
    \item $(0^{k+1}1^{k+1}\cdots n^{k+1})^e$ for $e \in \{id, rev \}$.
    \item $(012 \cdots n)^{e_1}(012 \cdots n)^{e_2}$ for $e_1, e_2 \in \{id, rev \}$. 
\end{enumerate}
Moreover, the choice of $n^6+1$ is best possible for every $k$.
\end{prop}

In both settings, it would also be interesting to find other families of unavoidable patterns. 

\acknowledgements
\label{sec:ack}
This project started at the 2025  Permutation Patterns Pre-Conference Workshop, hosted by the University of St.\ Andrews, Scotland and funded by the International Science Partnerships Fund (ISPF) and the UK Research and Innovation [EP/Y000609/1]. We thank Reed Acton and Victoria Ironmonger for useful discussions during the workshop.

\nocite{*}
\bibliographystyle{abbrvnat}
\bibliography{refs}

@article{Erdos1935,
  author  = {Erd\H{o}s, Paul and Szekeres, George},
  year    = {1935},
  pages   = {463-470},
  title   = {A combinatorial problem in geometry},
  volume  = {2},
  journal = {Compositio Mathematica}
}

@article{Mor1984,
  author  = {Mor, Moshe and Fraenkel, Aviezri S.},
  year    = {1984},
  pages   = {101–112},
  title   = {{C}ayley permutations},
  volume  = {48},
  journal = {Discrete Mathematics},
  doi     = {10.1016/0012-365X(84)90136-5}
}

@incollection{Vatter2015,
  title     = {Permutation Classes},
  author    = {Vatter, Vincent},
  booktitle = {Handbook of Enumerative Combinatorics},
  pages     = {777--858},
  year      = {2015},
  publisher = {Chapman and Hall/CRC}
}

@article{Schensted1961, 
title={Longest Increasing and Decreasing Subsequences}, 
volume={13}, 
DOI={10.4153/CJM-1961-015-3}, 
journal={Canadian Journal of Mathematics}, 
author={Schensted, Craige}, 
year={1961}, 
pages={179–191}
}

@article{Itskovich2019,
author = {Itskovich, Elizabeth J. and Levit, Vadim E.},
year = {2019},
month = {11},
pages = {237},
title = {What Do a Longest Increasing Subsequence and a Longest Decreasing Subsequence Know about Each Other?},
volume = {12},
journal = {Algorithms},
doi = {10.3390/a12110237}
}

@article{Burstein2003,
author = {Burstein, Alexander and Hästö, Peter and Mansour, Toufik},
year = {2003},
month = {11},
pages = {R20},
title = {Packing Patterns into Words},
volume = {9},
journal = {The Electronic Journal of Combinatorics},
doi = {10.37236/1692}
}

@article{Chung1980,
author = {Chung, F.R.K.},
year = {1980},
month = {11},
pages = {267–279},
title = {On unimodal subsequences},
volume = {29},
journal = {Journal of Combinatorial Theory, Series A},
doi = {10.1016/0097-3165(80)90021-7}
}

@article{Xu2024,
      title={On $k$-modal subsequences}, 
      author={Zijian Xu},
      year={2024},
      journal={arxiv:2403.13686}, 
}

@article{Gong2025,
author = {Gong, Charles},
year = {2025},
title = {Three Generalizations of {E}rdős {S}zekeres: $k$-Modal Subsequences},
journal = {arxiv:2508.20360},
}

@article{Suk2016,
author = {Suk, Andrew},
year = {2017},
month = {04},
pages = {1047-1053},
title = {On the {E}rdős--{S}zekeres convex polygon problem},
volume = {30},
journal = {Journal of the American Mathematical Society},
doi = {10.1090/jams/869}
}

\end{document}